\documentclass[reqno,11pt]{amsart}

\usepackage{graphicx}

\usepackage[ansinew]{inputenc}
\usepackage{amsfonts,epsfig}
\usepackage{latexsym}
\usepackage{mathabx}
\usepackage{amsmath}
\usepackage{amssymb}

\usepackage{color}

\newtheorem{theorem}{Theorem}
\newtheorem{lemma}[theorem]{Lemma}
\newtheorem{corollary}[theorem]{Corollary}

\newtheorem{proposition}[theorem]{Proposition}

\theoremstyle{definition}

\newtheorem{example}{Example}

\theoremstyle{remark}

\numberwithin{equation}{section}

\newcommand{\B}{\mathcal{B}}
\newcommand{\D}{\mathbb{D}}
\newcommand{\DD}{\widehat{\mathcal{D}}}
\newcommand{\Dd}{\widecheck{\mathcal{D}}}

\newcommand{\DDD}{\mathcal{D}}

\newcommand{\N}{\mathbb{N}}
\newcommand{\R}{\mathbb{R}}

\newcommand{\C}{\mathbb{C}}

\renewcommand{\phi}{\varphi}

\newcommand{\T}{\mathbb{T}}

\newcommand{\whw}{\widehat{\omega}}
\def\VMOA{\mathord{\rm VMOA}}
\def\BMOA{\mathord{\rm BMOA}}

\def\at{(\varphi_t)_{t\geq0}}
\def\Ct{(C_t)_{t\geq0}}

\def\A{A^p_{\omega}}
\def\a{A_{\omega,m}^p}              \def\g{\gamma}
           
     \def\om{\omega}      
            \def\f{\frac}
                  \def\z{\zeta}
                  
\def\G{\Gamma}

\renewcommand{\H}{\mathcal{H}}

\newenvironment{Prf}{\noindent{\emph{Proof of}}}
{\hfill$\Box$ }

\begin{document}
\title[Spectra of infinitesimal generators]{Spectra of infinitesimal generators of
composition semigroups on weighted Bergman
spaces induced by doubling weights}

\today

\thanks{The first author was supported by NNSF of China (No.12226312, 12226306) and Lingnan Normal University (No. LT2410). The second author is supported in part by Academy of Finland 356029. The third author is supported by NNSF of China (No. 12371131).}

\begin{abstract}
Suppose $(C_t)_{t\geq0}$ is the composition semigroup induced by a one-parameter semigroup $(\varphi_t)_{t\geq0}$ of analytic self-maps of the unit disk. The main purpose of the paper is to investigate the spectrum of the infinitesimal generator of $(C_t)_{t\geq0}$ acting on the weighted Bergman space induced by doubling weights, provided $(\varphi_t)_{t\geq0}$ is elliptic. The method applied is a certain spectral mapping theorem and a characterization of the spectra of certain composition operators. Eventual norm-continuity of $(C_t)_{t\geq0}$ also plays an important role, which can be depicted in terms of studying the difference of two distinct composition operators. As a byproduct, we also characterize a certain compact integral operator that is closely related to the resolvent of the infinitesimal generator of $(C_t)_{t\geq0}$. 
\end{abstract}

\keywords{Spectrum, Composition semigroup, Weighted Bergman space, Doubling weights, Difference of composition operators, Spectral mapping theorem, Eventually norm-continuous.}

\subjclass[2010]{30H20, 47B33, 47D06, 47A10}

\author[Ruishen Qian]{Ruishen Qian}
\address{School of Mathematics and Statistics, Lingnan Normal University, Zhanjiang 524048, Guangdong, China
}
\email{qianruishen@sina.cn}

\author[Fanglei Wu]{Fanglei Wu}
\address{Department of Physics and Mathematics, University of Eastern Finland, P.O.Box 111, 80101 Joensuu, Finland}
\email{fanglei.wu@uef.fi}
\email{fangleiwu1992@gmail.com}

\author[Hasi Wulan]{Hasi Wulan}
\address{Department of Mathematics, Shantou University, Shantou, 515063, People's Republic of China}
\email{wulan@stu.edu.cn}

\maketitle

\section{Introduction and main results}

A family $(\varphi_{t})_{t\geq0}$ of analytic self-maps of the unit disk $\D=\{z\in\C:|z|<1\}$ in the complex plane $\mathbb{C}$ is said to be a semigroup if the following conditions hold.
\begin{itemize}
\item[(i)] $\varphi_{0}$ is the identity map of $\mathbb{D}$;
\item[(ii)] $\varphi_{t}\circ\varphi_{s}=\varphi_{t+s}$, for $t,s\geq0$;
\item[(iii)] for each $z\in\mathbb{D}$, $\varphi_{t}(z)\rightarrow z$, as $t\rightarrow0^+$.
\end{itemize}
A semigroup $\at$ is said to be trivial if each $\varphi_t$ is the identity of $\D$. For any non-trivial semigroup $\at$, the infinitesimal generator of $\at$ is defined as the function
$$
G(z)=\lim_{t\to0^+}\f{\varphi_t(z)-z}{t},\quad z\in\D.
$$ 
There exist a point $b\in\overline{\D}$ and an analytic function $P:\D\mapsto\C$ with $\text{Re}\,P\geq 0$ such that
\begin{equation}\label{G}
G(z)=(\overline{b}z-1)(z-b)P(z).
\end{equation}
See \cite{BP} for the details. The representation \eqref{G} is unique and the point $b$ is said to be the Denjoy-Wolff point of $\at$. In general, we say that $\at$ is elliptic if its Denjoy-Wolff point $b$ belongs to $\D$. In this case, there exists a univalent function $h$, called Koenigs function, on $\D$ with $h(b)=0$ and $h'(b)=1$ such that
\begin{equation}\label{koenigs}
h(\varphi_t(z))=e^{G'(b)t}h(z),\, z\in\D,\, t\geq0.
\end{equation}
If $b\in\T$, the boundary of $\D$, then the non-tangential derivative $\varphi_t'(b)$ always exists and $\varphi_t'(b)\in(0,1]$ for every $t>0$. Moreover, $\at$ is said to be hyperbolic if $\varphi_t'(b)<1$ and parabolic if $\varphi_t'(b)=1$ for every $t>0$. For $b\in\T$, there exists a univalent function $h$, called Koenigs function, on $\D$ with $h(b)=0$ such that
$$
h(\varphi_t(z))=h(z)+t,\, z\in\D, \,t\geq0.
$$

Each semigroup $(\varphi_{t})_{t\geq0}$ gives rise to a semigroup $(C_{t})_{t\geq0}$ consisting of composition operators on $\H(\D)$, the set of analytic functions on $\D$, where
$$
C_{t}(f)=C_{\varphi_t}f=f\circ\varphi_{t}, \quad f\in \H(\mathbb{D}).
$$
The composition semigroup $(C_t)_{t\geq0}$ is said to be strongly continuous on a Banach space $X$ if $C_t$ is bounded on $X$ and
$$
 \lim_{t\rightarrow0^{+}}\|C_t x-x\|_{X}=0 \quad \mbox{for}~~\mbox{all}~~x\in X.
$$

The problem of studying the strong continuity of composition semigroup can be traced back to the initial work of E. Berkson and H. Porta \cite{BP} in which it was shown that every composition semigroup is strongly continuous on Hardy spaces. The same results were obtained on weighted Bergman spaces with standard weights and Dirichlet space, see \cite{AG2,AG3} and references therein. The aforementioned literature also proved that the (infinitesimal) generator $\Gamma$ of the composition semigroup $\Ct$ on these spaces has the form:
\begin{equation}\label{generator}
\G(f)=Gf',
\end{equation}
where $G$ is the infinitesimal generator of $\at$. 

It is well-known that the generator of a strongly continuous semigroup is closed. The classical Hille--Yosida theorem implies that the spectrum of the (infinitesimal) generator
of a strongly continuous semigroup always lies in a proper left half-plane, but it seems that there is no standard method to find out the complete distribution of the spectrum for a concrete generator. In general, the problem of describing the spectra of generators of strongly continuous composition semigroups is a tough question and remains unsettled in many cases.

The existing literature, see \cite{AG1,AG2} for example, offers several descriptions of the point spectrum of composition semigroups acting on Hardy spaces and Bergman spaces with the standard weights. To be concrete, given a semigroup $\at$ with the infinitesimal generator $G$, the Denjoy-Wolff point $b$ and Koenigs function $h$, the point spectrum of the generator of the corresponding composition semigroup $\Ct$, denoted by $P\sigma(\Gamma)$, can be characterized as follows:
\begin{itemize}
\item[(i)] If $b\in\D$, then $P\sigma(\G)=\{kG'(b):h^k\in X, k=0,1,2,3\ldots\}$;
\item[(ii)] If $b\in\T$, then $P\sigma(\G)=\{\lambda G(0):e^{\lambda h}\in X\}$,
\end{itemize}
where $X$ can be chosen as the Hardy space and the standard weighted Bergman space. The method used to find out the point spectrum of $\G$ depends on the argument principle. Using the harmonic measure, D. Betsakos \cite{Be1, Be2} provided an intuitive characterization of the point spectrum of the generator $\Gamma$ on the Hardy space and standard weighted Bergman space in the case of elliptic or hyperbolic semigroup $\at$.

We are mainly interested in the question of characterizing the spectrum of the generator of composition semigroups acting on weighted Bergman spaces induced by a class of weights admitting a certain doubling property. A non-negative function $\om\in L^1(\D)$ such that $\om(z)=\om(|z|)$ for all $z\in\D$ is called a radial weight. For $0<p<\infty$ and such an $\omega$, the Lebesgue space $L^p_\om$ consists of complex-valued measurable functions $f$ on $\D$ such that
$$
    \|f\|_{L^p_\omega}^p=\int_\D|f(z)|^p\omega(z)\,dA(z)<\infty,
    $$
where $dA(z)=\frac{dx\,dy}{\pi}$ is the normalized area measure on $\D$. The corresponding weighted Bergman space is $A^p_\om=L^p_\omega\cap\H(\D)$. As usual, we write $A^p_\alpha$ for the classical weighted Bergman spaces induced by the
standard weight $\om(z)=(\alpha+1)(1-|z|^2)^\alpha$ with $-1<\alpha<\infty$. Throughout this paper we assume $\widehat{\om}(z)=\int_{|z|}^1\om(s)\,ds>0$ for all $z\in\D$, for otherwise $A^p_\om=\H(\D)$. A radial weight $\om$ belongs to the class~$\DD$ if there exists a constant $C=C(\om)\ge1$ such that the tail integral $\widehat{\om}$ satisfies the doubling condition $\widehat{\om}(r)\le C\widehat{\om}(\frac{1+r}{2})$ for all $0\le r<1$. Moreover, if there exist $K=K(\om)>1$ and $C=C(\om)>1$ such that $\widehat{\om}(r)\ge C\widehat{\om}\left(1-\frac{1-r}{K}\right)$ for all $0\le r<1$, then we write $\om\in\Dd$. In other words, $\om\in\Dd$ if there exist $K=K(\om)>1$ and $C'=C'(\om)>0$ such that
	\begin{equation}\label{Dcheck}
	   \widehat{\om}(r)\le C'\int_r^{1-\frac{1-r}{K}}\om(t)\,dt,\quad 0\le r<1.
	\end{equation}
The intersection $\DD\cap\Dd$ is denoted by $\DDD$. It is well-known that for a radial weight $\om$ the norm convergence in $A^p_\om$ implies the uniform convergence in compact subsets of $\D$. In particular, the point evaluation functionals are bounded on $A^p_\om$. If further $\om\in\DD$, then a classical Hardy-Littlewood inequality yields
\begin{equation}\label{eqx1}
   |f(z)|\lesssim \frac{\|f\|_{\A}}{((1-|z|)\whw(z))^{1/p}},\quad z\in\D, ~~f\in\A. 
\end{equation}
See, for example \cite[p.35]{PeRa2014}.
Therefore, for a fixed $z\in\D$, the point evaluation functional $\delta_z$ satisfies $\|\delta_z\|\lesssim \frac{1}{((1-|z|)\whw(z))^{1/p}}$.
In fact, the other way around in the above estimate is correct as well. We just note that there exists a sufficiently large $\gamma=\gamma(\om)>0$ such that for any
$a\in\DD$ the function $f_{a,p,\gamma}(z)=\left(\frac{1-|a|^2}{1-\overline{a}z}\right)^{(\gamma+1)/p}$ belongs to $\A$ with $\|f_{a,p,\gamma}\|^p_{\A}\asymp \whw(a)(1-|a|).$ Then the definition of $\delta_z$ simply gives the other direction.
Therefore, we actually have
\begin{equation}\label{eq:functional}
 \|\delta_z\|\asymp \frac{1}{((1-|z|)\whw(z))^{1/p}},\quad z\in\D.
\end{equation}
Apart from this, we refer to \cite{pel2016,PeRa2014} and references therein for basic properties of these weighted Bergman spaces.

A possible and promising method to depict the spectrum of a strongly continuous composition semigroup is the so-called spectral mapping theorem which states the relationship between the spectrum of every $C_t$ and the spectrum of $\G$. Recall that a composition semigroup $(C_t)_{t\geq0}$ and its generator $\Gamma$ satisfy the spectral mapping theorem {(SMT)} if the following identity is valid:
\begin{equation}\label{weakspectrum}
\sigma(C_t)\backslash\{0\}=\overline{e^{t\sigma(\G)}}:=\overline{\{e^{t\lambda}:\lambda\in\sigma(\G)\}}
\quad \mbox{for}~~ t\geq0,
\end{equation}
where $\sigma(\cdot)$ is the spectrum of a certain linear operator. The identity \eqref{weakspectrum} naturally holds if $\Ct$ is acting on a Hilbert space, and hence it is valid for $A^2_\om$. However, in general case, \eqref{weakspectrum} holds only if $\sigma$ is replaced by the point spectrum $P\sigma$ or the residual spectrum $R\sigma$. In fact, it is only the approximate point spectrum $A\sigma$ that is responsible for the failure of (SMT). In this case, we only have the inclusion relation in the right direction of \eqref{weakspectrum}. Nevertheless, (SMT) can be guaranteed if $\Ct$ happens to be eventually norm continuous. See \cite{EN1} for more information about (SMT).

In this paper, we will investigate the spectrum of the generator of the eventually norm continuous composition semigroup $\Ct$ induced by an elliptic semigroup $\at$ on the weighted Bergman space $A^p_\om$ with $\om\in\DD$. We say that a strongly continuous composition semigroup $\Ct$ is eventually norm
 continuous on a Banach space $X$ if there exists a $t_0>0$ such that the composition $C_t$ is norm continuous for all $t>t_0$, i.e.
$$\lim_{s\to t}\|C_s-C_t\|_X=0.$$ From now on, denote by $\|T\|_X$, $r_X(T)$ and $\sigma_X(T)$ the operator norm, the spectral radius, and the spectrum of $T$ on $X$. Denote by $\|T\|_{e,X}$ and $r_{e,X}(T)$ the essential norm and essential spectral radius of the linear operator $T$ on the Banach space $X$. Our first main result can be stated as follows.

\begin{theorem}\label{theorem1}
Let $1\leq p<\infty$ and $\omega\in\widehat{\mathcal{D}}$. Suppose $\at$ is a semigroup of analytic self-maps of $\D$ with Denjoy-Wolff point $b\in\D$, infinitesimal generator $G$, and associated Koenigs function $h$. Denote by $\Gamma$ the infinitesimal generator of the corresponding composition semigroup $(C_t)_{t\geq0}$ on $A^p_\om$. Then the following statements hold:
\begin{itemize}
\item[(i)] If $\at$ consists of automorphisms of $\D$, then
$$
\sigma(\Gamma)=P\sigma(\Gamma)=\{kG'(b): h^k\in A^p_\om, k=0,1,2,3\ldots\};
$$
\item[(ii)] If no element of $\at$ is an automorphism of $\D$ and $\Ct$ is eventually norm continuous, then for any $t>0$, the spectrum $\sigma(\G)$ satisfies the following equality:
$$
\{e^{t\lambda}: \lambda\in\sigma(\G)\}\cup\{0\}=\{\lambda\in\C:|\lambda|\leq r_{e, A^p_\om}(C_{\varphi_{t}})\}\cup\{e^{kG'(b)t}:k=0,1,2,3\ldots\}.
$$

\end{itemize}
\end{theorem}

We make several observations on the above theorem. First of all, the most basic problem is to show that $\Ct$ is strongly continuous on $A^p_\om$ provided $\om\in\DD$. Otherwise, we can not proceed to consider the further problem of investigating the spectrum of the generator $\Gamma$ of $\Ct$. However, we will show that $\Ct$ is strongly continuous on $A^p_\om$ and, without any surprise, its generator has the same format as in the case of Hardy and standard Bergman spaces. Second, without any surprise, the point spectrum $P\sigma(\Gamma)$ obtained in (i) is actually the point spectrum of $\Gamma$ for any $\at$ with $b\in\D$, the method used to depict the point spectrum of the $\G$ is the argument principle, which also allows us to get $P\sigma(\G)$ when the Denjoy-Wolff point $b$ is on the boundary. Indeed, if $b\in\T$, then we can get
$P\sigma(\G)=\{\lambda G(0): e^{\lambda h(z)}\in A^p_{\omega}\}.$ Further, the reason why Theorem \ref{theorem1} is divided into two cases is that any element of the semigroup $\at$ is simultaneously an automorphism or not, according to the well-known result in \cite{A}. Finally, the interested readers may also realize that the second statement of Theorem \ref{theorem1} requires the precondition that $\Ct$ is eventually continuous on $A^p_\om$, but the assumption can be dropped when $p=2$ since the (SMT) naturally holds.

{
It is worth underlining that the point spectrum of $\Gamma$ can be described more transparently. Observing that $h$ is a certain spiral-like function and the key is to determine when $h^k$ belongs to $A^p_\om$, we first need to characterize spiral-like functions in $A^p_\om$, provided $\om\in\DD$. We say that a univalent function $h:\D\rightarrow\Omega:=h(\D)$ is said to be $\mu$-spiral-like if $\Omega$ is a $\mu$-spiral-like domain, which means that there is a number $\mu\in\C$ with $\text{Re}\mu>0$ such that for each $z\in\Omega$ and $t\geq0$ the point $e^{-t\mu}z$ also belongs to $\Omega$. See more information about spiral-like functions in \cite{Sd}. Given a $\mu$ with $\text{Re}\mu>0$ and $\theta_0, \eta\in[0,2\pi]$, the $\mu$-spiral sector with opening $\eta$ and the center angle $\theta_0$ is the set
$$
S_\mu(\theta_0,\eta)=\left\{e^{i\theta}e^{-t\mu}: ~t\in\R, ~|\theta-\theta_0|<\frac{\eta}{2}\right\}.
$$
If a $\mu$-spiral-like domain $\Omega$ contains a $\mu$-spiral sector, then it contains a maximal $\mu$-spiral sector of the form $S_\mu(\theta_0,\eta(\Omega))$, in the sense that there is no $\theta_1\in[0,2\pi]$ and no $\eta>\eta(\Omega)$ such that $S_\mu(\theta_1,\eta)\subset\Omega$. The number $\eta(\Omega)$ is called the maximal angular opening of $\Omega$. Now, we may depict spiral-like functions in $A^p_\om$.

\begin{proposition}\label{proo}
   Let $0< p<\infty$, $\omega\in\widehat{\mathcal{D}}$, and let $\mu\in\C$ with $\text{Re}\mu>0$. Suppose $f$ is a $\mu$-spiral-like function and let $\eta$ be the maximal angular opening of $f(\D)$. Then the following statements hold:
   \begin{itemize}
       \item [(i)] If $\eta=0$, then $f\in A^p_\om$;
       \item[(ii)] If $\eta>0$, then $f\in A^p_\om$ if and only if 
       \begin{equation}\label{condition}
       \int_0^1\frac{\whw(r)}{(1-r)^{p\eta\cos^2\alpha/\pi}}\,dr<\infty,
       \end{equation}
       where $\alpha=\arg\mu$.
   \end{itemize}
\end{proposition}

The proof relies on a well-known asymptotic equality of univalent functions $f$ in the Hardy spaces $H^p$:
$$
\|f\|^p_{H^p}\asymp\int_0^1 M_\infty^p(r,f)\,dr, 
$$
where $M_\infty(r,f)=\max_{|z|=r}|f(z)|$. See \cite[p.411]{HL} and \cite{Pra1927} for the proof. This immediately gives the estimate for univalent function $f$ in weighted Bergman space $A^p_\nu$ with the radial weight $\nu$:
\begin{equation}\label{univalent}
\|f\|^p_{A^p_\nu}\asymp\int_0^1 M_\infty^p(r,f)\left(\int_r^1\nu(t)t\,dt\right)\,dr.
\end{equation}
Together with two growth theorems established in \cite{Be2}, we may prove the above proposition, which, together with Theorem~\ref{theorem1}, gives the following more precise characterization of the point spectrum $P\sigma(\Gamma)$.

\begin{corollary}\label{theorem+}
    Let $1\leq p<\infty$ and $\omega\in\widehat{\mathcal{D}}$. Suppose $\at$ is a semigroup of analytic self-maps of $\D$ with Denjoy-Wolff point $b\in\D$, infinitesimal generator $G$, and associated Koenigs function $h$. Denote by $\Gamma$ the infinitesimal generator of the corresponding composition semigroup $(C_t)_{t\geq0}$ on $A^p_\om$. Let $\mu=-G'(b)$, $\alpha=\arg\mu$, and let $\eta$ be the maximal angular opening of the $\mu$-spiral-like domain $h(\D)$. Then we have 
    \begin{itemize}
        \item [(i)] If $\eta=0$, then 
        $$
        P\sigma(\Gamma)=\{kG'(b): k=0,1,2,3\ldots\};
        $$
        \item[(ii)] If $\eta>0$, then
         $$
        P\sigma(\Gamma)=\{kG'(b): k=0,1,2,3\ldots,k_0\},
        $$
    where 
    $$
    k_0=\max\left\{k\in\N: k<\int_0^1\frac{\whw(r)}{(1-r)^{p\eta\cos^2\alpha/\pi}}\,dr\right\}.
    $$
    \end{itemize}
\end{corollary}
}

As aforementioned, the key approach is the spectral mapping theorem. Since each element of the semigroup $\at$ is univalent on $\D$, to use (SMT), we need first characterize the spectrum of the composition $C_\varphi$, provided $\varphi$ is univalent and has a fixed point in $\D$. The approach we apply here is inherited from the argument of Cowen and MacCluer \cite{CM1} in which the case of unweighted Bergman spaces was dealt with. The key method applied there is using so-called iteration sequence, which will also be employed in the proof of Theorem \ref{theorem1}. Another important ingredient of the proof is the useful identity concerning essential spectral radii $r_{e,A^p_\om}(C_\varphi)$ for distinct $1\leq p<\infty$, which can be achieved through a shortcut instead of using the old method that requires many laborious calculations.

Now, another natural question is when the composition semigroup $\Ct$ is eventually continuous on $A^p_\om$ provided $\om\in\DD$. The most straightforward way to characterize the eventual continuity of $\Ct$ is estimating the norm of the operator $C_\varphi-C_\psi$, the difference of $C_\varphi$ and $C_\psi$, when $\varphi$ and $\psi$ are two distinct self-maps on $\D$.
In the paper, as the second main result, we can get the characterization if the weight involved is confined to belong to the class $\DDD$.  To state the result, we define $\phi_a(z)=\frac{a-z}{1-\overline{a}z}$ for all $a,z\in\D$. The pseudohyperbolic distance between two points $a$ and $z$ in $\D$ is $\delta(a,z)=|\varphi_a(z)|$. For $a\in\D$ and $0<r<1$, the pseudohyperbolic disc of center $a$ and of radius $r$ is $\Delta(a,r)=\{z\in \D:\delta(a,z)<r\}$. It is well known that $\Delta(a,r)$ is an Euclidean disk centered at $(1-r^2)a/(1-r^2|a|^2)$ and of radius $(1-|a|^2)r/(1-r^2|a|^2)$.

\begin{theorem}\label{theorem2}
   Let $1\leq p<\infty$, $\omega\in\DDD$. Suppose $\at$ is a semigroup of analytic self-maps of $\D$ and $\Ct$ is the corresponding composition semigroup on $A^p_\om$. Then $\Ct$ is an eventually norm-continuous operator semigroup on $A^p_\om$ if and only if there exist $\gamma_0=\gamma_0(\om)>0$  and $t_0>0$ such that for any $\gamma>\gamma_0$ and $t,s>t_0$
\begin{equation}\label{normcontinuous}
\begin{split}
&\lim_{s\to t}\sup_{a\in\D}\frac{(1-|a|)^{\gamma}}{\whw(a)}\int_{\D}\left(\f{\delta^p(\varphi_t(z),\varphi_s(z))}{|1-\bar{a}\varphi_t(z)|^{\gamma+1}}+\f{\delta^p(\varphi_t(z),\varphi_s(z))}{|1-\bar{a}\varphi_s(z)|^{\gamma+1}}\right)\om(z)\,dA(z)=0.
\end{split}
\end{equation}
\end{theorem}

The quantity on the left of \eqref{normcontinuous} without limit is an estimate of $\|C_{\varphi_t}-C_{\varphi_s}\|^p$. We will actually prove that the similar estimate remains hold for general $C_\varphi-C_\psi$ on $A^p_\om$ with $\om\in\DDD$ and any different analytic self-maps $\varphi$ and $\psi$ on $\D$. Although $C_\varphi-C_\psi$ is automatically bounded due to the subordination principle, the above estimate of the norm $C_\varphi-C_\psi$ requires the so-called Carleson measure. Recall that for a space of analytic functions $X$, a positive measure $\mu$ is said to be $p$-Carleson measure for $X$ if the identity operator $I_d: X\to L^p_\mu$ is bounded. It follows from \cite[Theorem 2]{LiuRattya} that a positive $\mu$ is a $p$-Carleson measure for $A^p_\om$ with $\om\in\DDD$ if and only if there exists $r=r(\om)\in(0,1)$ such that 
\begin{equation}\label{Carlson}
    \sup_{z\in\D}\frac{\mu(\Delta(z,r))}{\om(\Delta(z,r))}<\infty,
\end{equation}
where $\om(E)=\int_E\om(z)\,dA(z)$ for any $E\subset\D$.
The reason why the same method cannot be carried out in the case $A^p_\om$ with $\om\in\DD$ is that the same characterization \eqref{Carlson} is not valid anymore for those $ A^p_\om$, see \cite[Theoem 3.3]{pel2016} for example. Indeed, we will show that the same norm estimate of $C_\varphi-C_\psi$ for the case $\om\in\DDD$ is an upper bound for the norm $C_\varphi-C_\psi$ for the case $\om\in\DD$ but we do not know if it is also a lower bound or not. To get the general characterization for the eventual continuity of $\Ct$ on $A^p_\om$ with $\om\in\DD$, we stress here that an intriguing attempt is probably to combine the Carleson measure and the fact that every element of $\at$ is univalent. 

As explained above, we see that \eqref{normcontinuous} is a sufficient condition such that $\Ct$ is eventually norm continuous on $A^p_\om$ with $\om\in\DD$, which immediately implies that the class of $\at$ such that the corresponding $\Ct$ is norm continuous on those $A^p_\om$ is vast. To be more concrete, for a semigroup $\at$, if the exists a $t_0>0$ such that $\|\varphi_{t_0}\|_{H^\infty}<1$, then the corresponding $\Ct$ is eventually norm continuous on $A^p_\om$, provided $\om\in\DD$.

Another important regular property of $\Ct$ is eventual compactness, which says that there exists a $t_0>0$ such that for any $t>t_0$, $C_t$ is compact. It is well known that eventual compactness is stronger than eventual continuity. Indeed, \cite[Theorem 5.12]{EN1} indicates that the eventual compactness of a strongly continuous operator semigroup is equivalent to the eventual
continuity and the compactness of the resolvent (operator) of its infinitesimal generator. If the resolvent of the infinitesimal generator happens to be compact, then the spectrum of the generator only contains the point spectrum. Recall that the resolvent $R(\lambda, \Gamma)$ of $\Gamma$, the generator of composition semigroup $\Ct$ on the Banach space $X$ is defined by $(\lambda-\G)^{-1}$ for any $\lambda\in\rho(\G)=\C\setminus\sigma_X(\G)$, the resolvent set of $\G$. In general, it can be represented as 
$$
R(\lambda, \Gamma)(f)=\int_0^{\infty} e^{-\lambda t} C_t(f) \,dt, \quad f \in X.
$$ 
In particular, if the Denjoy-Wolff point of the induced semigroup $\at$ point is 0, then the resolvent of $\G$ has a neatly concrete representation:
$$
R(\lambda, \Gamma) f(z)=-\frac{1}{G^{\prime}(0)} \frac{1}{(h(z))^{-\frac{\lambda}{G^{\prime}(0)}}} \int_0^z f(\zeta)(h(\zeta))^{-\frac{\lambda}{G^{\prime}(0)}-1} h^{\prime}(\zeta) d \zeta,~~f\in X,
$$
where $G$ is the infinitesimal generator of $\at$ and $h$ is the corresponding Koenigs function. 

Now, we return to the weighted Bergman space $A^p_\om$ with $\om\in\DD$. We will show that $-G'(0)$ belongs to $\rho(\Gamma)$, which means that $R(-G'(0),\Gamma)$ is bounded. It is well known that $h$ is a certain spiral-like function, which is certainly univalent. However, the linear operator 
$$
R_h(f)(z)=\frac{1}{h(z)}\int_0^z f(\z) h'(\z)\,d\z\quad f\in\H(\D)
$$
is well-defined for a general univalent function $h$ with $h(0)=0$. This enables us to consider the boundedness and compactness on $A^p_\om$ with $\om\in\DD$ for the full range $0<p<\infty$ and for a general univalent function $h$ with $h(0)=0$. As a byproduct, we have the following result:

\begin{theorem}\label{theorem3}
Let $0<p<\infty$ and $\om\in\DD$. Suppose $h$ is a univalent function on $\D$ satisfying $h(0)=0$. Then the following statements are equivalent:
\begin{itemize}
    \item [(i)] $R_h$ is compact on $A^p_\om$;
    \item[(ii)] $\log\frac{h(z)}{z}\in\B_0$;
    \item[(iii)] $\log\frac{h(z)}{z}\in\VMOA$.
\end{itemize}
\end{theorem}

The proof of the theorem absorbs from the method provided in \cite{AG4}, where it was used to characterize the compactness of the resolvent of the infinitesimal generator on Hardy spaces. The key point of this approach is to show that $R_h$ has `similar' properties as another novel linear operator $Q_h$, defined by 
$$
Q_h(f)(z)=\frac{1}{z}\int_0^z f(\z)\,d\z+\frac1z\int_0^z f(\z)\z\left(\log\frac{h(\z)}{\z}\right)^{\prime}\,d\z,\quad f\in\H(\D).
$$
Since $h$ is univalent with $h(0)=0$, an application of \cite[Theorem 2]{Ba} implies that $\log\frac{h(z)}{z}$ belongs to $\BMOA$. Then it follows from \cite[Proposition 6.1 and Theorem 6.8]{pel2016} that $R_h$ is automatically bounded on $A^p_\om$ with $\om\in\DD$. Likewise, the compactness part can be deduced similarly. 

The paper is organized as follows. In Section \ref{sec2}, we first characterize the spectra of some composition operators on $A^p_\om$, then using this characterization, we prove Theorem \ref{theorem1}. In Section \ref{sec3}, we characterize the eventual norm-continuity of $\Ct$ in terms of estimating the norm of the difference of two distinct composition operators on $A^p_\om$. Section \ref{sec4} contains the proof of Theorem \ref{theorem3} as well as several other discussions about the operator $R_h$.

\vskip 0.3cm
We finish the introduction with a couple of words about the notation used in this paper. Throughout the paper, the letter $C=C(\cdot)$ will denote an absolute constant whose value depends on the parameters indicated
in the parenthesis and may change from one occurrence to another. If there exists a constant
$C=C(\cdot)>0$ such that $a\le Cb$, then we write either $a\lesssim b$ or $b\gtrsim a$. In particular, if $a\lesssim b$ and
$a\gtrsim b$, then we denote $a\asymp b$ and say that $a$ and $b$ are comparable.

\section{Spectra of composition operators on $A^p_\om$}\label{sec2}

Suppose $\omega\in\widehat{\mathcal{D}}$ and $z\in\D\backslash\{0\}$. Let
$$
\omega^\star(z)=\int_{|z|}^1 s\omega(s)\log\frac{s}{|z|}\,ds.
$$
From \cite[Lemma 2.1]{pel2016}, we know
\begin{equation}\label{omstar}
    \om^\star(z)\asymp\whw(z)(1-|z|),\quad |z|>\frac12.
\end{equation}
In this section, we focus on the spectrum of a certain composition operator on weighted Bergman spaces with doubling weights. To be precise, the composition operator we consider here is confined to be univalent but not an automorphism with the fixed point in the interior of $\D$. The approach applied here is inherited from \cite{MS} in which the spectrum of the same composition operator on unweighted Bergman spaces was investigated. It is pointed out that a certain closed subspace of weighted Bergman spaces is supposed to be an efficient tool. Therefore, for a weight $\om$ and a fixed $m\in\N$, we define $A^p_{\om,m}=\{f\in A^p_\om: f(z)=z^m g(z), g\in A^p_\om\}$. In other words, $A^p_{\om,m}$ consists of all $f\in A^p_\om$ that has a zero of order at least $m$ at 0.

We need the following growth estimate of functions in $\a$. From now on, for any $f\in\H(\D)$, we set 
 $$
    M_p(r,f)=\left (\frac{1}{2\pi}\int_0^{2\pi}
    |f(re^{i\theta})|^p\,d\theta\right )^{\frac{1}{p}},\quad 0<p<\infty,
    $$
and
    $$
    M_\infty(r,f)=\max_{0\le\theta\le2\pi}|f(re^{i\theta})|.
    $$

\begin{lemma}\label{lemma: growth}
Let $0<p<\infty$ and $\omega\in\widehat{\mathcal{D}}$. For any $m\in\N$ there exists an $r_m$ such that
$$
|f(z)|\lesssim\frac{m|z|^m}{((1-|z|)\widehat{\omega}(r_m)\widehat{\omega}(z))^{1/p}}\|f\|_{\A},\quad f\in\a,~~z\in\D.$$
\end{lemma}

\begin{proof}
  When $m\geq2$, take $r_m=(\frac{1}{m})^\frac{1}{m}$. Then the definition of $\a$ indicates that for any $f\in\a$, there exists a $g\in\A$ such that $f(z)=z^m g(z)$.
By  the monotonicity of $M_p(r,f)$, we have
\begin{align*}
\|g\|_{A_\omega^p}^p&= \int_0^{r_m} M_p^p(r,g)\omega(r)rdr+\int_{r_m}^1 M_p^p(r,g)\omega(r)rdr\\
&\leq \int_0^{r_m} M_p^p(r_m,g)\omega(r)rdr+\int_{r_m}^1 M_p^p(r,g)\omega(r)rdr\\
&\leq\frac{\int_0^{r_m}\omega(r)rdr}{\int_{r_m}^1\omega(r)rdr}M_p^p(r_m,g)\int_{r_m}^1\omega(r)rdr+\int_{r_m}^1 M_p^p(r,g)\omega(r)rdr\\
&\leq \left(\frac{\int_0^{r_m}\omega(r)rdr}{\int_{r_m}^1\omega(r)rdr}+1\right)\int_{r_m}^1 M_p^p(r,g)\omega(r)r\,dr\\
&\leq r_m^{-mp}\left(\frac{\int_0^{1}\omega(r)rdr}{\int_{r_m}^1\omega(r)rdr}\right)\int_{r_m}^1 M_p^p(r,g)\omega(r)r^{mp+1}dr
\lesssim \f{m^p}{\widehat{\omega}(r_m)}\|f\|_{A_\omega^p}^p,
\end{align*}
where the last inequality holds because $\om$ is a radial weight. This together with \eqref{eqx1} gives
$$
|f(z)|=|z^m||g(z)|\lesssim\frac{|z|^m\|g\|_{A^p_\om}}{((1-|z|)\whw(z))^{1/p}}\lesssim\frac{m|z|^m\|f\|_{\A}}{((1-|z|)\widehat{\omega}(r_m)\widehat{\omega}(z))^{1/p}},
$$
which is the exact inequality we are aiming for.

When $m=1$, let $r_m=1/2$. By repeating the same steps above, we get the same result. The proof is complete.
\end{proof}

\begin{lemma}\label{le:functional3}
  Let $1\leq p<\infty$ and $\omega\in\widehat{\mathcal{D}}$. Then for any $m\in\mathbb{N}$, we have
$$\|\delta_z\|_{\a}\leq\|\delta_z\|_{\A}\lesssim 2^m\|\delta_z\|_{\a}, \quad|z|\geq\frac{1}{2}.$$
\end{lemma}

\begin{proof}
It is clear that $\|\delta_z\|_{\A}\geq\|\delta_z\|_{\a}$. To prove the right inequality, for a fixed $z\in\D$, consider the function $g(\z)=g_{z,p,\g,m}(\z)=\z^m \left(\frac{1-|z|^2}{1-\overline{z}\z}\right)^{(\gamma+1)/p}.$ Then for some large enough $\gamma>0$, $g\in\a$. Thus, the definition of $\delta_z$, \eqref{eq:functional}, and Lemma \ref{lemma: growth} yield
$$
\|\delta_z\|_{\a}\gtrsim \frac{|z|^m}{\left((1-|z|)\widehat{\omega}(z)\right)^{1/p}}\gtrsim\frac{1}{2^m}\|\delta_z\|_{\A},\quad |z|\geq\frac{1}{2}.
$$
The proof is completed.
\end{proof}

The following result shows a close relation of the essential spectral radius of composition operators on different weighted Bergman spaces. Probably a classical proof originates from \cite{MS} in which only the case of $1<p<\infty$ was solved, while the case $p=1$ was missing. However, we will provide a trivial method that allows us to deal with the case $1\leq p<\infty$ directly, hence avoiding lots of laborious calculations. Indeed, the approach applied here relies on an asymptotic estimate of the essential spectral radius of a composition operator on the weighted Bergman space, which was demonstrated in \cite[Theorem 19]{PR2016}.

\begin{lemma}\label{essentialradius}
     Let $1\leq p<\infty$, $\omega\in\widehat{\mathcal{D}}$,  and let $\varphi$ be an analytic self-map of $\D$. Then
$$\left(r_{e,\A}(C_{\varphi})\right)^p=\left(r_{e,A_\omega^2}(C_{\varphi})\right)^2.$$
\end{lemma}

\begin{proof}
   Using \cite[Theorem 19]{PR2016} directly with the special case $p=q$, we immediately see that $\|C_\varphi\|^p_{\A}\asymp\|C_\varphi\|^2_{A^2_\om}$. By the definition of essential spectral radius $r_{e,\A}(C_{\varphi})=\lim_{n\rightarrow\infty}(\|C_{\varphi}^n\|_{e,\A})^{\frac1n}$ and the fact
$C_{\varphi_{n}}=C_{\varphi}^n$, replacing $\varphi$ by its $n$th iteration $\varphi_n$, we have
$$
\|C_{\varphi_{n}}\|_{e,\A}^p\asymp\|C_{\varphi_{n}}\|_{e,A_{\om}^2}^2, \quad 1\leq p<\infty.
$$
This easily leads to the identity we desire.
\end{proof}

The proof of Theorem \ref{theorem1} relies on the so-called iteration sequence. We say that $\{z_k\}$ is an iteration sequence if there exists $\varphi:\D\rightarrow\D$ such that $\varphi(z_k)=z_{k+1}$. The following lemma plays a critical role in our proof whose proof can be found in \cite[Lemmas 7.34 and 7.35]{CM}.
\begin{lemma}\label{le:iteration}
  Let $\varphi$ be an analytic self-map of $\D$, not an
automorphism and $\varphi(0)=0$. Given $0<r<1$, there exists $1\leq M_r<\infty$ such that if $\{z_k\}_{-K}^\infty$ is an iteration sequence for $\varphi$ with $|z_n|\geq r$ for some non-negative
integer $n$ and $\{w_k\}_{k=-K}^n$ are arbitrary complex numbers, then there exists an $f\in H^\infty$ satisfying
$$f(z_k)=w_k, \quad -K\leq k\leq n$$
and
$$\|f\|_{H^\infty}\leq M_r\sup\{|w_k|:-K\leq k\leq n\}.$$
Moreover, there exists $0<b<1$ such that for any iteration sequence $\{z_k\}$ we have
$$\frac{|z_{k+1}|}{|z_k|}\leq b,$$
whenever $|z_k|\leq \frac{1}{2}.$
\end{lemma}

The above lemma enables us to simplify some notation. For an iteration sequence  $\{z_k\}_{-K}^\infty$, from now on, suppose $|z_0|\geq 1/2$ and
$n=\max\{k: |z_k|\geq 1/4\}.$ It follows from Lemma \ref{le:iteration} that there is a $1/2\leq b<1$ such that for any $k>n$,
$\frac{|z_{k+1}|}{|z_k|}\leq b$ and hence
$$|z_k|\leq |z_n|b^{k-n}, \quad k\geq n.$$
Now, we are ready to  
characterize the spectrum of a certain composition operator that will be used to prove Theorem ~\ref{theorem1}.

\begin{theorem}\label{spectrum}
    Let $1\leq p<\infty$ and $\om\in\DD$. Suppose $\varphi$ is a univalent self-map on $\D$ with the fixed point $a\in\D$ but not an automorphism. Then
    $$
    \sigma_{A^p_\om}(C_\varphi)=\{\lambda\in\C: |\lambda|<r_{e,A^p_\om}(C_\varphi)\}\cup\{(\varphi'(a))^n\}^\infty_{n=0}.
    $$
\end{theorem}

\begin{proof}
 Without loss of generality,  assume that $a=0$. Since all polynomials are contained in weighted Bergman spaces, it follows from the proof of \cite[Lemma 2]{K} that $z^n$ does not belong to the range of $(C_\varphi-(\varphi'(0))^nI)$, and hence $\{(\varphi'(0))^n\}_{n=0}^\infty$ is in the spectrum of $C_\varphi$.
If $|\lambda|>r_{e,\A}(C_\varphi)$, by \cite[Proposition 2.2]{BS}, we see that $(C_\varphi-\lambda I)$ is Fredholm of index 0, which means that it is either invertible or has non-trivial
kernel. Therefore if further $\lambda\in\sigma_{\A}(C_\varphi)$, then $\lambda$ is an eigenvalue of $C_\varphi$. Moreover, Koenigs Theorem (\cite[Theorem 2.63]{CM}) tells that the eigenvalues of $C_\varphi$ must be of form $(\varphi'(0))^n$. Therefore, to finish the proof, it suffices to verify
\begin{align}\label{0109-1}
\{\lambda\in\mathbb{C}: |\lambda|\leq r_{e,\A(C_{\varphi})}\}\subseteq\sigma_{\A}(C_\varphi).
\end{align}

Trivially, $0\in\sigma_{\A}$, which implies that (\ref{0109-1}) holds when $r_{e,\A(C_{\varphi})}=0$.
Now, we assume $r_{e,\A(C_{\varphi})}>0$. We only need to deal with the case when $0<|\lambda|<r_{e,\A(C_{\varphi})}$ since spectrum is closed.  If we can prove that $(C_\varphi-\lambda I)^*$ is not invertible, then it will be done.  Let now $C_m$ be the restriction of $C_\varphi$ acting on $\a$. Note that $\varphi(0)=0$ implies that $\a$ is an invariant subspace of $C_\varphi$. Then the same proof employed in \cite[Lemma 7.17]{CM} (in our case, we only need to take $\mathcal{K}=\overline{span \{1, z, z^2,\cdots, z^{m-1}\}}$, and $\mathcal{L}=\a$) shows that if $(C_\varphi-\lambda I)$ is invertible on $A^p_\om$ then $(C_m-\lambda I)$ must be invertible on $\a$, which simply yields $\sigma_{\a}(C_m)\subseteq\sigma_{\A}(C_\varphi)$.
 Therefore, to show that $(C_\varphi-\lambda I)^*$ is not invertible on $A^p_\om$, it suffices to prove that $(C_m-\lambda I)^*$ is not bounded below.

Let now $r=1/4$. By Lemma \ref{le:iteration}, we can choose two constants $b$ and $M$ satisfying $1/2\leq b<1$ and  $1\leq M<\infty$. Fix $m$ sufficiently large such that
$$\frac{b^m}{|\lambda|}<\frac{1}{SM}.$$
Here, $S$ is large enough and will be fixed later. For any iteration sequence $\{z_k\}_{-K}^\infty$, we can define a linear functional $L_{\lambda}$ on $\a$ by
$$
L_\lambda(f)=\sum_{k=-K}^{\infty}\lambda^{-k}f(z_k).
$$
Indeed, by the hypothesis that $|z_k|<1/4$ when $k>n$ and Lemma \ref{lemma: growth}, we deduce that
\begin{equation}\nonumber
\begin{split}
\left|\sum_{k=n+1}^{\infty}\lambda^{-k}f(z_k)\right|\lesssim&\sum_{k=n+1}^{\infty}|\lambda|^{-k}\frac{m|z_k|^m}{((1-|z_k|)\widehat{\omega}(r_m)\widehat{\omega}(z_k))^{1/p}}\|f\|_{\A}
\\\lesssim&\f{m\|f\|_{\A}}{(\widehat{\omega}(r_m))^\f{1}{p}}\sum_{k=n+1}^{\infty}\frac{|z_k|^m}{|\lambda|^k}.
\end{split}
\end{equation}
Here, $r_m=(\f{1}{m})^\f{1}{m}$.
However, since $|z_k|\leq|z_n|b^{k-n}$ and $\frac{b^m}{|\lambda|}<1$,  $L_{\lambda}$ is a bounded linear functional on $\a$. Therefore, for every $f\in \a$, we have 
$$
(C_m^*-\lambda)L_\lambda f=L_\lambda  (C_m-\lambda)f=L_\lambda(f\circ\varphi)-\lambda L_\lambda f=-\lambda^{K+1}f(z_{-K}),
$$
which gives
$$
(C_m^*-\lambda)L_\lambda=-\lambda^{K+1}\delta_{z_{-K}}.
$$
Now, for a sufficiently large $\gamma$ depending on $\om$, let $\{w_k\}_{k=-K}^n$ be the sequence such that $w_k=0$ when $k\neq0$ and $-K\leq k<n$; and when $k=0$ and $k=n$, $w_k=e^{i\theta_k}$, where $\theta_k=-\arg\left(\frac{z_k^m}{\lambda^k(1-\overline{z_0}z_k)^{\frac{\gamma+1}{p}}}\right)$.  Applying this $\{w_k\}_{k=-K}^n$ in Lemma \ref{le:iteration}, we see that exists $h\in H^\infty$ with $\|h\|_{H^\infty}\leq M$ such that:
\begin{enumerate}
  \item[(i)] When $k=0$ and $k=n$, $|h(z_k)|=1$.
  \item[(ii)]  When $k=0$ and $k=n$, $\frac{z_k^m h(z_k)}{\lambda^k(1-\overline{z_0}z_k)^{\frac{\gamma+1}{p}}}>0$.
  \item[(iii)]When $k\neq0$ and $-K\leq k<n$, $h(z_k)=0$.
\end{enumerate}
Therefore, for the above sufficiently large $\gamma$, we have
$$
L_\lambda\left(\frac{z^m(1-|z_0|^2)^{\frac{\gamma}{p}}h(z)}{(\widehat{\omega}(z_0))^{1/p}(1-\overline{z_0}z)^{\frac{\gamma+1}{p}}}\right)
=\sum_{k=-K}^{\infty}\lambda^{-k}\frac{z_k^m(1-|z_0|^2)^{\frac{\gamma}{p}}h(z_k)}{(\widehat{\omega}(z_0))^{1/p}(1-\overline{z_0}z_k)^{\frac{\gamma+1}{p}}}
=\sum_{k=-K}^\infty J_{k,h}.
$$
Note that $J_{k,h}(-K\leq k\leq n)$ are zeros unless $k=0,n$. Also,
$$
J_{0,h}=\frac{|z_0|^m}{(\widehat{\omega}(z_0))^{1/p}(1-|z_0|^2)^{\frac{1}{p}}}
$$
and
$$J_{n,h}=\frac{|z_n|^m(1-|z_0|^2)^{\frac{\gamma}{p}}}{|\lambda|^n|(\widehat{\omega}(z_0))^{1/p}|1-z_0z_n|^{\frac{\gamma+1}{p}}}
\geq\frac{|z_n|^m(1-|z_0|^2)^{\frac{\gamma}{p}}}{2^{\frac{\gamma+1}{p}}|\lambda|^n(\widehat{\omega}(z_0))^{1/p}}.$$
Since
$$\sum_{k=n+1}^{\infty}\left(\frac{b^m}{|\lambda|}\right)^{k-n}<\frac{\frac{1}{SM}}{1-\frac{1}{SM}}\leq\frac{1}{(S-1)M},$$
we obtain
\begin{equation}\nonumber
\begin{split}
\left|\sum_{k=n+1}^\infty J_{k,h}\right|
=&\left|\sum_{k=n+1}^{\infty}\lambda^{-k}\frac{z_k^m(1-|z_0|^2)^{\frac{\gamma}{p}}h(z_k)}{(\widehat{\omega}(z_0))^{1/p}(1-\overline{z_0}z_k)^{\frac{\gamma+1}{p}}}\right|
\\\leq&(\frac{4}{3})^{\frac{\gamma+1}{p}}\frac{M|z_n|^m(1-|z_0|^2)^{\frac{\gamma}{p}}}{|\lambda|^n(\widehat{\omega}(z_0))^{1/p}}\sum_{k=n+1}^{\infty}\left(\frac{b^m}{|\lambda|}\right)^{k-n}
\\\leq&(\frac{4}{3})^{\frac{\gamma+1}{p}}\cdot\frac{1}{S-1}\cdot\frac{|z_n|^m(1-|z_0|^2)^{\frac{\gamma}{p}}}{|\lambda|^n(\widehat{\omega}(z_0))^{1/p}}
\\\leq&(\frac{1}{3})^{\frac{\gamma+1}{p}}\cdot\frac{|z_n|^m(1-|z_0|^2)^{\frac{\gamma}{p}}}{|\lambda|^n(\widehat{\omega}(z_0))^{1/p}}.
\end{split}
\end{equation}
Here, we choose a large enough constant $S>0$ such that $\frac{4^\f{\gamma+1}{p}}{S-1}<1$. Therefore,
\begin{align*}
\left|L_\lambda\left(\frac{z^m(1-|z_0|^2)^{\frac{\gamma}{p}}h}{(\widehat{\omega}(z_0))^{1/p}(1-\overline{z_0}z)^{\frac{\gamma+1}{p}}}\right)\right|
&\geq J_{0,h}+J_{n,h}-\left|\sum_{k=n+1}^\infty J_{k,h}\right| \\
&\geq\frac{|z_0|^m}{(\widehat{\omega}(z_0))^{1/p}(1-|z_0|^2)^{\frac{1}{p}}}.
\end{align*}
Write $g(z)=\frac{(1-|z_0|^2)^{\frac{\gamma}{p}}}{(\widehat{\omega}(z_0))^{1/p}(1-\overline{z_0}z)^{\frac{\gamma+1}{p}}}\in \A.$
Then we have
$$
\left\|\frac{z^m(1-|z_0|^2)^{\frac{\gamma}{p}}h}{(\widehat{\omega}(z_0))^{1/p}(1-\overline{z_0}z)^{\frac{\gamma+1}{p}}}\right\|_{\a}
=\|z^mgh\|_{\a}\leq CM,
$$
and further,
\begin{equation*}
\|L_{\lambda}\|_{\a} \geq\frac{|L_\lambda(z^m gh)|}{CM}
\gtrsim\frac{|z_0|^m}{(\widehat{\omega}(z_0))^{1/p}(1-|z_0|^2)^{\frac{1}{p}}}\asymp\|\delta_{z_0}\|_{\a}.
\end{equation*}
We are now ready to make a judicious choice of iteration sequence. Recall that
$$r_{e,\A}(C_\varphi)=\lim_{n\rightarrow\infty}\|C_{\varphi_n}\|_{e,\A}^{1/n}=\inf_{n\in\mathbb{N}}\|C_{\varphi_n}\|_{e,\A}^{1/n}.$$
Since $\varphi$ is univalent, by Proposition 18 and Theorem 19 in \cite{PR2016}, we have
$$\|C_{\varphi_n}\|_{e,A_\omega^2}^2=\limsup_{|w|\rightarrow 1^-}\frac{\omega^\star(\varphi_n^{-1}(w))}{\omega^\star(w)}=\limsup_{|w|\rightarrow 1^-}\frac{\omega^\star(w)}{\omega^\star(\varphi_n(w))}.$$
For any $\lambda$ with $0<|\lambda|<r_{e,\A}(C_\varphi)$,  choose  $\tau$ satisfying $|\lambda|<\tau<r_{e,\A}(C_\varphi)$. Applying Lemma \ref{essentialradius}, we deduce
$$r_{e,\A}(C_\varphi)=\lim_{n\rightarrow\infty}\|C_{\varphi_n}\|_{e,\A}^{1/n}=\lim_{n\rightarrow\infty}\left(\|C_{\varphi_n}\|_{e,A_\omega^2}^{2/p}\right)^{1/n}.$$
That is, there exists $N$, such that for any $n>N$,
$$
\|C_{\varphi_n}\|_{e,A_\omega^2}^{2/p}>\tau^n.
$$
Thus, for any $K\geq N$, we can choose $w$ approaching  the boundary of $\D$ such that
\begin{enumerate}
  \item[(i)]$\left(\frac{\omega^\star(|w|)}{\omega^\star(|\varphi_K(w)|)}\right)^{1/p}\geq\tau^K$.
  \item[(ii)] $|\varphi_K(w)|\geq 1/2$.
  \item [(iii)]$\frac{\|\delta_{\varphi_K(w)}\|_{\a}}{\|\delta_w\|_{\a}}
      \gtrsim\frac{1}{2^m}\frac{\|\delta_{\varphi_K(w)}\|_{\A}}{\|\delta_w\|_{\A}}.$
\end{enumerate}
Combining (i), (iii), \eqref{omstar} and Lemma \ref{le:functional3}, we deduce that
$$
\frac{\|\delta_{\varphi_K(w)}\|_{\a}}{\|\delta_w\|_{\a}}\gtrsim \frac{1}{2^m}\tau^K.
$$
With this choice of $w$, we get an iteration sequence $\{z_k\}_{k=-K}^{\infty}$ by setting $z_{-K}=w$ so that $z_0=\varphi_K(w)$ and $|z_0|\geq 1/2$. At this point, the positive integer $K$ is still
arbitrary, except for the requirement that $K\geq N$. Our estimates say
\begin{equation}\nonumber
\begin{split}
\frac{\|(C^*_m-\lambda)L_{\lambda}\|_{\a}}{\|L_{\lambda}\|_{\a}}\lesssim&\frac{|\lambda|^{K+1}\|\delta_{z_{-K}}\|}{\|\delta_{z_{0}}\|}
\lesssim|\lambda|\left(\frac{\lambda}{\tau}\right)^K.
\end{split}
\end{equation}
If $K\geq N$ large enough, then $\left(\frac{\lambda}{\tau}\right)^K\rightarrow 0$, which means that $(C^*_m-\lambda)$ is not bounded below. The proof is completed.
\end{proof}

We first assert that each semigroup $\at$ induces a strongly continuous composition semigroup on $A^p_\om$, provided $1\leq p<\infty$ and $\om\in\DD$. Although the proof is standard, we still present it for the sake of the completeness of the statement.

\begin{proposition}
   Let $1\leq p<\infty$ and $\omega\in\widehat{\mathcal{D}}$. Suppose $\at$ is a semigroup of analytic self-maps of $\D$ with the infinitesimal generator $G$. Then the corresponding composition semigroup $\Ct$ is strongly continuous on $A^p_{\omega}$ with the infinitesimal generator $\G$:
\begin{equation}\label{Gamma}   
\G f=Gf'
\end{equation}
on the domain
$$
D(\G)=\{f\in A^p_{\omega}: Gf'\in A^p_{\omega}\}.
$$
Moreover, $\Ct$ is uniformly continuous on $A^p_{\omega}$ if and only if $\at$ is trivial.
\end{proposition}

\begin{proof}
Since $\omega\in\DD$, the polynomials are dense in $A^p_{\omega}$. Therefore for any $f\in A^p_{\omega}$, there exists a sequence of polynomials $\{P_n\}$ such that $\lim_{n\rightarrow\infty}\|f-P_n\|_{A^p_{\omega}}=0$. It follows from triangle inequality that
\begin{align*}
\|C_tf-f\|_{A^p_{\omega}}&\leq\|C_tf-C_tP_n\|_{A^p_{\omega}}
+\|C_tP_n-P_n\|_{A^p_{\omega}}+\|f-P_n\|_{A^p_{\omega}}\\
&\leq(\|C_t\|+1)\|f-P_n\|_{A^p_{\omega}}+\|C_tP_n-P_n\|_{A^p_{\omega}}.
\end{align*}
According to \cite[Theorem 15 and (4.7)]{PR2016}, we obtain $\sup_{t\in[0,1]}\|C_t\|<\infty$. Therefore, to prove $\lim_{t\rightarrow0^+}\|C_tf-f\|_{A^p_{\omega}}=0$, it suffices to prove $\lim_{t\rightarrow0^+}\|C_tP-P\|_{A^p_{\omega}}=0$ for each polynomial $P$. Equivalently we only need to show that for each $n\geq0$ $\lim_{t\rightarrow0^+}\|(\varphi_t)^n-e_n\|_{A^p_{\omega}}=0$, where $e_n(z)=z^n$. For a fixed $n\geq0$, since $|\varphi_t(z)-z^n|^p\leq(|z|^n+1)^p\in L^1_\om$ and $\varphi^n_t(z)\to z^n$ as $t\to 0^+$ for each $z\in\D$, a direct application of the Lebesgue dominant convergence theorem gives the desired result. Thus $\Ct$ is strongly continuous on $A^p_{\omega}$.

By definition, the domain of $\G$ is
$$
D(\G)=\{f\in A^p_{\omega}: \lim_{t\rightarrow0^+}\frac{C_t(f)-f}{t} ~~\mbox{exists~~in}~A^p_{\omega}\}.
$$
Let now $D=\{f\in A^p_{\omega}: Gf'\in A^p_{\omega}\}$. We are going to show that if $f\in D(\G)$, then $Gf'\in A^p_{\omega}$. Indeed, if $f\in D(\G)$, then $\G(f)\in A^p_{\omega}$ and
$$
\lim_{t\rightarrow0^+}\left\|\frac{C_tf-f}{t}-\G(f)\right\|_{A^p_{\omega}}=0.
$$
Since for a fixed $0<r<1$ the well known inequality $M_{\infty}(r,f)\lesssim M_{p}(\frac{1+r}{2},f)(1-r)^{-\frac1p}$ yields
$$
M_{\infty}^p(r,f)\lesssim\frac{\|f\|^p_{A^p_{\omega}}}{(1-r)\widehat{\omega}(r)},
$$
convergence in the norm of $A^p_{\omega}$ implies the pointwise convergence. Therefore, by the definition of the infinitesimal generator, we deduce for every $z\in\D$,
\begin{align*}
\G f(z)&=\lim_{t\rightarrow0^+}\frac{f(\varphi_t(z))-f(z)}{t}=
\lim_{t\rightarrow0^+}\frac{f(\varphi_t(z))-f(\varphi_0(z)}{t}\\
&=\frac{\partial(f(\varphi_t(z)))}{\partial t}\bigg|_{t=0}=f'(\varphi_t(z))\frac{\partial\varphi_t(z)}{\partial t}\bigg|_{t=0}=G(z)f'(z).
\end{align*}
That is, $G(z)f'(z)=\G f(z)\in A^p_{\omega}$ and $D(\G)\subseteq D$. On the other hand, for $\lambda\in\rho(\G)$, the resolvent set of $\G$, we have
$$
D=\{f\in A^p_{\omega}: Gf'\in A^p_{\omega}\}=\{f\in A^p_{\omega}: Gf'-\lambda f\in A^p_{\omega}\}
=R(\lambda,\G)(A^p_\om),
$$
where $R(\lambda,\G)=(\lambda I-\G)^{-1}$ is the resolvent operator of $\G$. Since $R(\lambda,\G)(A^p_{\omega})\subseteq D(\G)$, $D\subseteq D(\G)$. Hence, $D=D(\G)$.

If $\Ct$ is uniformly continuous on $A^p_{\omega}$, then the infinitesimal generator $\G$ is bounded on $A^p_{\omega}$. To show $\at$ is trivial, it is equivalent to show $G\equiv0$. To the end, we may consider polynomials $e_n(z)=z^n$. Then $\G(e_n)=nGe_{n-1}$, and taking $n=1$ we see that $G\in A^p_{\omega}$. Since
$\G$ is bounded on $A^p_{\omega}$, for $n\geq1$, we have $\|\G e_n\|\lesssim\|e_n\|$, that is,
\begin{equation*}
n^p\int_0^1M^p_p(r,G)r^{(n-1)p+1}\omega(r)dr\lesssim \int_0^1r^{np+1}\omega(r)dr.
\end{equation*}
It follows that $n^pM_p^p(\frac12,G)\lesssim1$. Since $G\in A^p_{\omega}$, we have $M_p(\frac12,G)=0$. Thus $G\equiv0$. The proof is complete.
\end{proof}

\Prf Theorem ~\ref{theorem1}. (i) If each $\varphi_t$ is an automorphism of the disk, then the resolvent of $\Gamma$ is compact. To show this, we can choose a suitable automorphism to conjugate and hence assume without loss of generality that $b=0$. Then each $\varphi_t$ is the rotation of $\D$. Let $\varphi_t(z)=e^{iat}z$ for real $a\neq0$. Then $G(z)=-iaz$, and $\Gamma(f)=-ia f'(z)$. On one hand, since $(ia-\Gamma)(f)=g$ has the unique analytic solution
$$
f(z)=\frac{1}{i\alpha z}\int_0^zg(\zeta)d\zeta.
$$ 
and the operator
$$
g\to\frac{1}{ia z}\int_0^zg(\zeta)d\zeta
$$
is compact on $A^p_{\omega}$ by, for example, a similar proof of \cite[Theorem 6.4]{pel2016}, $ia$ belongs to the resolvent set of $\Gamma$ and the resolvent 
$$
(iaI-\Gamma)^{-1}(g)(z)=\frac{1}{ia z}\int_0^zg(\zeta)d\zeta,\quad g\in A^p_\om
$$
is compact on $A^p_\om$. On the other hand, the well-known resolvent equation (see for example \cite[p.157]{EN1})
$$
R(\lambda,\G)-R(\mu,\G)=(\mu-\lambda)R(\lambda,\G)R(\mu,\G) \quad \lambda,\mu\in\rho(\G)
$$
shows that $R(\lambda,\G)$ is compact on $A^p_{\omega}$ for all $\lambda\in\rho(\G)$ if and only if it is compact for a certain $\lambda_0\in\rho(\G)$. Hence the spectrum of $\Gamma$ contains only the point spectrum by \cite[Corollary 1.15, p.162]{EN1}. (We underline here that a general criterion of compact resolvent is provided in Corollary~\ref{coroally}.)

Now, differentiating both sides of \eqref{koenigs} with respect to $t$ and letting $t=0$, we get that $G(z)=G'(b)\frac{h(z)}{h'(z)}$. On one hand,  suppose $f\in \H(\D)$ and $\lambda\neq0$ such that $\G f=\lambda f$. Then \eqref{Gamma} yields
$$
\frac{f'}{f}=\lambda\frac{h'}{G'(b)h}.
$$
Choosing $r$ such that $|b|<r<1$ and $f$ has no zeros on $|z|=r$, we deduce
$$
\frac1{2\pi i}\int_{|z|=r}\frac{f'(z)}{f(z)}dz=\frac{\lambda}{G'(b)}\frac1{2\pi i}\int_{|z|=r}\frac{h'(z)}{h(z)}\,dz.
$$
This together with the argument principle yields $\frac{\lambda}{G'(b)}=k$, where $k$ is any nonnegative integer. On the other hand, notice that the differential equation
$f'(z)=k\frac{h'(z)}{h(z)}f(z)$
has the solution $f(z)=ch^k$, $c\neq0$. This finishes the proof of (i).
\vskip3mm
(ii) The proof is pretty straightforward. Since we hypothesize that $\Ct$ is eventually norm continuous on $A^p_\om$, the (SMT) holds naturally. Therefore, by Theorem \ref{spectrum}, we get for any $t>0$
$$
\{e^{t\lambda}: \lambda\in\sigma(\G)\}=\{\lambda\in\C:|\lambda|\leq r_{e, A^p_\om}(C_{\varphi_{t}})\}\cup\{(\varphi_t'(b))^k\}:k=0,1,2,3\ldots\}\setminus\{0\}.
$$
The next step is taking the derivative of both sides of \eqref{koenigs} with respect to $z$ and taking $z=b$, we obtain 
$$
\varphi_t'(b)=e^{G'(b)t},\quad t>0.
$$
This together with the fact that point $0\in\sigma_{A^p_{\om}}(C_{\varphi_t})$ for any $t>0$, we can get the desired result.

\vskip1cm
{\Prf Proposition~\ref{proo}. By \cite[Lemma 4]{Be2}, for any $\varepsilon>0$, there exists a constant $C_\varepsilon$ such that 
$$
M_{\infty}(r,f)\leq\frac{C_\varepsilon}{(1-r)^{(\eta+\varepsilon)\cos^2\alpha/\pi}},\quad 0<r<1.
$$
Therefore, 
\begin{align*}
\int_0^1 M^p_\infty(r,f)\int_r^1\om(t)t\,dt\,dr&\asymp \int_0^1 M^p_\infty(r,f)\whw(r)\,dr \\
&\lesssim C_\varepsilon\int_0^1\frac{\whw(r)}{(1-r)^{p(\eta+\varepsilon)\cos^2\alpha/\pi}}\,dr.
\end{align*}
This together with \eqref{univalent} implies that $f\in A^p_\om$ if $\eta=0$ and $\varepsilon$ is small enough. If $\eta>0$, since the hypothesis \eqref{condition} implies that there exists a sufficiently small $\varepsilon$ such that 
$$
\int_0^1\frac{\whw(r)}{(1-r)^{p(\eta+\varepsilon)\cos^2\alpha/\pi}}\,dr<\infty,
$$
then the same reasoning as the above shows $f\in A^p_\om$. Conversely, if $f\in A^p_\om$, then \cite[Lemma 3]{Be2} and \eqref{univalent} yield \eqref{condition} directly. The proof is complete. 

\vskip0.5cm
\Prf Corollary~\ref{theorem+}. Since $h$ is $\mu$-spiral-like function,  the proof is straightforward by applying Theorem~\ref{theorem1} and Proposition \ref{proo}.
}
\section{Eventually norm continuous composition semigroup}\label{sec3}

The proof of Theorem \ref{theorem2} does essentially rely on the result of characterizing bounded difference of composition operators on weighted Bergman spaces. In general, we have the result for the case that $\om\in\DDD$. The proof of the sufficiency is basically similar to the proof of \cite[Proposition 4]{LRW}, while the necessity should be modified since the original proof shown in \cite[Proposition 6]{LRW} used a wrong estimate \cite[p.795]{SE}, see \cite{LSS} for their corrections.  

The Carleson square $S(I)$ based on an interval $I \subset \mathbb{T}$ is the set $S(I)=\left\{r e^{i t} \in\right.$ $\left.\mathbb{D}: e^{i t} \in I, 1-|I| \leq r<1\right\}$, where $|I|$ denotes the Lebesgue measure of $I \subset \mathbb{T}$. We associate to each $a \in \mathbb{D} \backslash\{0\}$ the interval $I_a=\left\{e^{i \theta}:\left|\arg \left(a e^{-i \theta}\right)\right| \leq \frac{1-|a|}{2}\right\}$, and denote $S(a)=S\left(I_a\right)$.

To prove the necessity, we need the following result, whose proof can be found in \cite[Theorem 2.8]{KooWang}.

\begin{lemma}\label{xxxJ}
    Suppose $s>1$ and $0<r_0<1$. Then there exist $N=N\left(r_0\right)>1$ and $C=C\left(s, r_0\right)>0$ such that
$$
\begin{aligned}
& \left|\frac{1}{(1-\bar{a} z)^s}-\frac{1}{(1-\bar{a} w)^s}\right|+\left|\frac{1}{\left(1-t_N \bar{a} z\right)^s}-\frac{1}{\left(1-t_N \bar{a} w\right)^s}\right| \\
& \quad \geq C \delta(z, w)\left|\frac{1}{(1-\bar{a} z)^s}\right|,
\end{aligned}
$$
for all $z \in \Delta\left(a, r_0\right)$ with $1-|a|<1 /(2 N), t_N=1-N(1-|a|)$ and $w \in \mathbb{D}$.
\end{lemma}

\begin{proposition}\label{difference}
    Let $0<p<\infty$, $\om\in\DDD$, and let $\nu$ be a positive Borel measure on $\D$. Suppose $\varphi$ and $\psi$ are analytic self-maps of $\D$. Let $\sigma(z)=\delta(\varphi(z),\psi(z))$. Then the difference operator $C_\varphi-C_\psi: A^p_\om\to L^p_\nu$ is bounded if and only if both weighted composition operators $\sigma C_\varphi$ and $\sigma C_\psi$ are bounded from $A^p_\om$ to $L^p_\nu$. Moreover, there exists a $\gamma_0=\gamma_0(\om)$ such that for all $\gamma>\gamma_0$
\begin{equation}
\|C_\varphi-C_\psi\|^p\asymp\sup_{a\in\D}\frac{(1-|a|)^{\gamma}}{\whw(a)}\int_{\D}\left(\f{\sigma(z)^p}{|1-\bar{a}\varphi(z)|^{\gamma+1}}+\f{\sigma(z)^p}{|1-\bar{a}\psi(z)|^{\gamma+1}}\right)\om(z)\,dA(z).
\end{equation}
\end{proposition}

\begin{proof}
Suppose $\sigma C_\varphi$ and $\sigma C_\psi$ are bounded. For a fixed $0<r<1$, let $E=\{z\in\D: \sigma(z)>r\}$. Then 
\begin{equation}\label{1}
    \begin{split}
        \|C_\varphi(f)-C_\psi(f)\|^p_{L^p_\nu}&=\left(\int_E+\int_{\D\setminus E}\right)|f(\varphi(z))-f(\psi(z))|^p\nu(z)\,dA(z)\\
       &\leq \frac{1}{r}\int_E \sigma(z)|f(\varphi(z))-f(\psi(z))|^p\nu(z)\,dA(z)+I\\
       &\leq\frac1r(\|\sigma C_\varphi f\|^p_{L^p_\nu}+\|\sigma C_\psi f\|^p_{L^p_\nu})+I,\quad f\in A^p_\om,
    \end{split}
\end{equation}
where $I=\int_{\D\setminus 
 E}|f(\varphi(z))-f(\psi(z))|^p\nu(z)\,dA(z)$. 
To estimate $I$, for a measurable function $h$ and an analytic self-map $\varphi$, we define the pull-back measure $\varphi_*(h)(M)=\int_{\varphi^{-1}(M)}h\,dA$
for each measurable set $M\subseteq\D$. Now, using the \cite[Lemma 1]{LiuRattya} and Fubini's theorem, we deduce that for any $r<R<1$
\begin{equation}\label{2}
    \begin{split}
    I&=\int_{\D\setminus 
 E}|f(\varphi(z))-f(\psi(z))|^p\nu(z)\,dA(z)\\
 &\lesssim \int_{\D\setminus 
 E}\frac{\sigma(z)^p}{(1-|\varphi(z)|)^2}\int_{\Delta(\varphi(z),R)}|f(\z)|^p\,dA(\z)\nu(z)\,dA(z)\\
 &\leq \int_{\D}\frac{|f(\z)|^p}{(1-|\z|)^2}\int_{\varphi^{-1}(\Delta(\z,R))}\sigma(z)^p\nu(z)\,dA(z)\,dA(\z)\\
 &=\int_\D|f(\z)|^p\frac{\varphi_*(\sigma^p\nu)(\Delta(\z,R))}{(1-|\z|)^2}\,dA(\z).
    \end{split}
\end{equation}
By the definition of the pull-back measure, we see that $\|\sigma C_\varphi f\|_{L^p_\nu}=\|f\|_{L^p_{\varphi_*(\sigma^p\nu)}}$. Therefore, for any $a\in\D$, it follows from the hypothesis $\om\in\DDD\subseteq\DD$ that 
\begin{equation}\label{3}
    \int_{S(a)}\frac{\varphi_*(\sigma^p\nu)(\Delta(\z,R))}{(1-|\z|)^2}\,dA(\z)\lesssim \varphi_*(\sigma^p\nu)(S(a))\lesssim\om(S(a)),\quad |a|>R.
\end{equation}
Combining \eqref{2}, \eqref{3} and the Carleson measure characterization for $A^p_\om$ \cite[Theorem 3.3]{pel2016}, we get 
\begin{equation*}
    I\lesssim \sup_{a\in\D}\frac{\varphi_*(\sigma^p\nu)(S(a))}{\om(S(a))} {\|f\|^p_{A^p_\om}}.
\end{equation*}
Since $\sigma C_\varphi$ and $\sigma C_\psi$ are bounded, both $\varphi_*(\sigma^p\nu)$ and $\psi_*(\sigma^p\nu)$ are $p$-Carleson measures for $A^p_\om$. \eqref{1}, the above estimate and \cite[Theorem 3.3]{pel2016} yield
\begin{equation}\label{x}
  \|C_\varphi-C_\psi\|^p\lesssim  \sup_{a\in\D}\frac{\varphi_*(\sigma^p\nu)(S(a))}{\om(S(a))}+\sup_{a\in\D}\frac{\psi_*(\sigma^p\nu)(S(a))}{\om(S(a))}.
\end{equation}
Then the upper bound can be acquired from \eqref{norm}, which will be proved a bit later.

\vskip3mm
Conversely, suppose $C_\varphi-C_\psi$ is bounded. For any $a\in\D$, consider the function 
$$
f_a(z)=\left(\frac{1-|a|}{1-\overline{a}z}\right)^{\frac{\gamma+1}{p}}\om(S(a))^{-\frac{1}{p}},\quad z\in\D.
$$
It follows from \cite[Lemma 2.1]{pel2016} that there exists $\gamma_0=\gamma_0(\om)$ such that for all $\gamma>\gamma_0$, $\|f_a\|_{A^p_\om}\asymp1$. 
A standard calculation based on \cite[Theorem 3.3]{pel2016} shows that if a positive measure $\mu$ is a p-Carleson measure for $A^p_\om$, then for those $\gamma$ the identity operator $Id: A^p_\om\to L^p_\mu$ satisfies 
\begin{equation}\label{norm}
  \|I_d\|^p_{A^p_\om\to L^p_\mu} \asymp\sup_{a\in\D}\|f_a\|^p_{L^p_\mu}.
\end{equation}
Now, fix a $\gamma>\gamma_0$ and an $r\in(0,1)$.  Then choose such $N=N(r)$ and $t_N$ appearing in Lemma ~\ref{xxxJ}, we have $\|f_{t_{N}a}\|\asymp1$. When $|a|>1-\frac{1}{2N}$, the hypothesis and Lemma \ref{xxxJ} yield
\begin{equation}\label{x1}
    \begin{split}
&\quad\|C_\varphi-C_\psi\|^p\gtrsim\|(C_\varphi-C_\psi)f_a\|_{L^p_\nu}^p+\|(C_\varphi-C_\psi)f_{t_{N}a}\|_{L^p_\nu}^p\\
&=\int_\D\left(\left|\left(\frac{1-|a|}{1-\overline{a}\varphi(z)}\right)^{\frac{\gamma+1}{p}}-\left(\frac{1-|a|}{1-\overline{a}\psi(z)}\right)^{\frac{\gamma+1}{p}}\right|^p\right.\\
&\quad\left.+\left|\left(\frac{1-|a|}{1-t_N\overline{a}\varphi(z)}\right)^{\frac{\gamma+1}{p}}-\left(\frac{1-|a|}{1-t_N\overline{a}\psi(z)}\right)^{\frac{\gamma+1}{p}}\right|^p\right)\frac{\nu(z)}{\omega(S(a))}\,dA(z)\\
&\gtrsim\int_{\varphi^{-1}(\Delta(a,r))}\left(\left|\left(\frac{1-|a|}{1-\overline{a}\varphi(z)}\right)^{\frac{\gamma+1}{p}}-\left(\frac{1-|a|}{1-\overline{a}\psi(z)}\right)^{\frac{\gamma+1}{p}}\right|\right.\\
&\quad\left.+\left|\left(\frac{1-|a|}{1-t_N\overline{a}\varphi(z)}\right)^{\frac{\gamma+1}{p}}-\left(\frac{1-|a|}{1-t_N\overline{a}\psi(z)}\right)^{\frac{\gamma+1}{p}}\right|\right)^p\frac{\nu(z)}{\omega(S(a))}\,dA(z)\\
&\gtrsim \frac{(1-|a|)^{\gamma+1}}{\om(S(a))}\int_{\varphi^{-1}(\Delta(a,r))}\frac{\sigma(z)^p}{|1-\overline{a}\varphi(z)|^{\gamma+1}}\nu(z)\,dA(z)\asymp\frac{\varphi_*(\sigma^p\nu)(\Delta(a,r))}{\om(S(a))}.
\end{split}
\end{equation}
The same step shows that for $\psi$, we have 
\begin{equation}\label{x2}
   \|C_\varphi-C_\psi\|^p\gtrsim \frac{\psi_*(\sigma^p\nu)(\Delta(a,r))}{\om(S(a))}.
\end{equation}
When $|a|\leq 1-\frac{1}{2N}$, then first $\om(S(a))$ is uniformly bounded and 
$$
\sigma(z)=\frac{|\varphi(z)-\psi(z)|}{|1-\overline{\varphi(z)}\psi(z)|}\lesssim |\varphi(z)-\psi(z)|,\quad \varphi(z)\in\Delta(a,r)~\text{or}~\psi(z)\in\Delta(a,r).
$$
Therefore, we deduce 
\begin{equation}\label{x3}
\begin{split}
    \frac{\varphi_*(\sigma^p\nu)(\Delta(a,r))}{\om(S(a))}&=\frac{1}{\om(S(a))}\int_{\varphi^{-1}(\Delta(a,r))}\left|\frac{\varphi(z)-\psi(z)}{1-\overline{\varphi(z)}\psi(z)}\right|^p\nu(z)\,dA(z)\\
    &\lesssim\int_\D|\varphi(z)-\psi(z)|^p\nu(z)\,dA(z)\lesssim\|C_\varphi-C_\psi\|^p.
    \end{split}
\end{equation}
Similarly,
\begin{equation}\label{x4}
  \frac{\psi_*(\sigma^p\nu)(\Delta(a,r))}{\om(S(a))}\lesssim\|C_\varphi-C_\psi\|^p.  
\end{equation}
Combining \eqref{x1}, \eqref{x2}, \eqref{x3},  and \eqref{x4}, we see that for any $a\in\D$, 
\begin{equation}\label{xxx}
\|C_\varphi-C_\psi\|^p\gtrsim \frac{\varphi_*(\sigma^p\nu)(\Delta(a,r))}{\om(S(a))}+\frac{\psi_*(\sigma^p\nu)(\Delta(a,r))}{\om(S(a))}.
\end{equation}
Now, since $\om\in\DDD$, we see that there exist $r=r(\om)$ such that for any $z\in\D$, $\om(S(z))\asymp\om(\Delta(z,r))$. This together with \eqref{Carlson} yields a positive $\mu$ is a Carleson measure for $A^p_\om$ if and only if 
$$
\sup_{z\in\D}\frac{\mu(\Delta(z,r))}{\om(\Delta(z,r))}<\infty
$$
and moreover 
$$
\|Id\|^p_{A^p_\om\to L^p_\mu}\asymp\sup_{z\in\D}\frac{\mu(\Delta(z,r))}{\om(\Delta(z,r))}, 
$$
which combined with \eqref{xxx} and \eqref{norm} complete the proof.
\end{proof}

\Prf~Theorem~\ref{theorem2}. By the definition of the eventual continuity and a direct application of Proposition \ref{difference}, we are in a position to get the result. \hfill$\Box$

\vskip 3mm
We underline here that the proof of the sufficiency only requires the hypothesis that $\om\in\DD$ instead of $\om\in\DDD$, which means that any semigroup $\at$ satisfying \eqref{normcontinuous} induces an eventually continuous composition semigroup on $A^p_\om$, provided $\om\in\DD$. This fact guarantees the existence of such $\at$ and the amount is huge.  

\begin{corollary}
    Let $1\leq p<\infty$, $\omega\in\DD$. Suppose $\at$ is a semigroup of analytic self-maps of $\D$ and $\Ct$ is the corresponding composition semigroup on $A^p_\om$. If there exists a $t_0$ such that $\|\varphi_{t_0}\|_{H^\infty}<1$, then $\Ct$ is eventually norm continuous on $A^p_\om$. 
\end{corollary}

\begin{proof}
    By the hypothesis and the property of semigroup $\at$, we see that $\|\varphi_t\|_{H^\infty}<1$ for all $t>{t_0}$. Therefore, we deduce 
    $$
    \lim_{s\to t}\delta(\varphi_s(z),\varphi_t(z))=\lim_{s\to t}\left|\frac{\varphi_s(z)-\varphi_t(z)}{1-\overline{\varphi_s(z)}\varphi_t(z)}\right|=0,\quad z\in\D,~~~s,t>t_0.
    $$
    Then the application of Theorem \ref{theorem2} (sufficiency holds for the case $\om\in\DD$) and the dominated convergence theorem verify the result, as required. 
\end{proof}

Another regularity property of a strongly continuous semigroup we are considering now is the so-called eventual compactness. Recall that the composition semigroup $\Ct$ induced by $\at$ is eventually compact on a Banach space $X$ if there exists a $t_0>0$ such that for all $t>t_0$, $C_{\varphi_t}$ is compact on $X$. It follows from \cite[Lemma 5]{EN1} that an eventually compact semigroup is always eventually norm continuous. We have the following description for the eventual compactness of composition semigroup on $A^p_\om$, provided $\om\in\DD$. 

\begin{theorem}
   Let $1\leq p<\infty$, $\omega\in\DD$. Suppose $\at$ is a semigroup of analytic self-maps of $\D$ and $\Ct$ is the corresponding composition semigroup on $A^p_\om$. Then $\Ct$ is eventually compact on $A^p_\om$ if only if there exists a $t_0>0$ such that for all $t>t_0$
   $$
\limsup_{|z|\rightarrow 1^-}\frac{\omega^\star(z)}{\omega^\star(\varphi_t(z))}=0.
   $$
\end{theorem}

\begin{proof}
    Since each $\varphi_t$ is univalent on $\D$, a direct application of the combination of \cite[Corollary 21]{PR2016} and \cite[Lemma 23]{PR2016} completes the proof.
\end{proof}

One has seen that the eventual compactness of an operator semigroup implies its eventual norm continuity, while the converse does not hold anymore. Indeed, it turns out that the eventual compactness of an operator semigroup is equivalent to its eventual norm continuity and additionally, the compactness of its generator, see \cite[Lemma 5.11]{EN1} for the proof. In general, if the generator of a strongly continuous operator semigroup has a compact resolvent, then this semigroup has some very nice properties. For instance, the spectrum of the generator only consists of the point spectrum. In particular, for a composition semigroup on weighted Bergman spaces with doubling weights, we can also characterize the compactness of the resolvent of its generator, which will be presented in the next section by studying a certain integral operator.

\section{A novel integral operator on $A^p_\om$}\label{sec4}

For a strongly continuous composition semigroup $\Ct$ on a Banach space $X$ with the generator $\Gamma$, and for any $\lambda\in\rho(\G)$, the resolvent $R(\lambda,\G)$ of $\G$  can be generally represented as 
$$
R(\lambda, \Gamma)(f)=\int_0^{\infty} e^{-\lambda t} C_t(f) \,dt, \quad f \in X.
$$
However, if the Denjoy-Wolff point of the induced semigroup $\at$ is 0, then $R(\lambda,\G)$ has a concrete representation: 

\begin{lemma}\label{lemma3}
Let $1\leq p<\infty$, $\omega\in\DD$, and let $\at$ be a non-trivial semigroup of self-maps on $\D$ with Denjoy-Wolff point 0, infinitesimal generator $G$ and Koenigs function $h$. Suppose $\Ct$ is the corresponding composition semigroup on $A^p_{\omega}$, with the generator $\Gamma$. Then for all $\lambda\in\rho(\Gamma)$ the resolvent operator of $\Gamma$ has the following representation:
\begin{equation*}
R(\lambda,\Gamma)f(z)=-\frac{1}{G'(0)}\frac{1}{(h(z))^{\frac{\lambda}{-G'(0)}}}\int_{0}^zf(\zeta)
(h(\zeta))^{\frac{\lambda}{-G'(0)}-1}h'(\zeta)d\zeta.
\end{equation*}
In particular, $-G'(0)$ belongs to $\rho(\Gamma)$ and hence
\begin{equation}\label{re}
R(-G'(0),\Gamma)f(z)=-\frac{1}{G'(0)h(z)}\int_{0}^zf(\zeta)h'(\zeta)d\zeta.
\end{equation}
\end{lemma}

\begin{proof}
Let
$$
R:=-\frac{1}{G'(0)}\frac{1}{(h(z))^{\frac{\lambda}{-G'(0)}}}\int_{0}^zf(\zeta)
(h(\zeta))^{\frac{\lambda}{-G'(0)}-1}h'(\zeta)d\zeta.
$$
It is elementary to compute that $(\lambda I-\Gamma)R=R(\lambda I-\Gamma)=I$, which shows that $R$ is the resolvent operator of $\Gamma$. Since the Denjoy-Wolff point of $\at$ is 0, it is easy to see that $\text{Re}\,(-G'(0))\geq0$. If $\text{Re}\,(-G'(0))>0$, by \cite[Theorem 15 and (4.7)]{PR2016}, we deduce that there exist constants $\eta=\eta(\om)>0$ and $C=C(p,\om,\eta)>0$ such that $\|C_t\|\leq C\left(\frac{1+|\varphi_t(0)|}{1-|\varphi_t(0)|}\right)^\eta$. Since $0$ is the Denjoy-Wolff point, we immediately know that $\|C_t\|$ is uniformly bounded, and hence we have
$$
\omega_0:=\lim_{t\rightarrow\infty}\frac{\log\|C_t\|}{t}=0.
$$
So $-G'(0)\in\rho(\Gamma)$ by \cite[Theorem 1.10, p.42]{EN1}. If $-G'(0)$ is a pure imaginary number, write $G(z)=-i\alpha z$, where $\alpha\in\mathbb{R}\setminus\{0\}$ and $\Gamma f=-i\alpha zf'(z)$. In this case, $(i\alpha-\Gamma)(f)=g$ has the unique analytic solution
$$
f(z)=\frac{1}{i\alpha z}\int_0^zg(\zeta)d\zeta.$$
It is not difficult to see that the operator
$$
g\to\frac{1}{i\alpha z}\int_0^zg(\zeta)d\zeta
$$
is bounded on $A^p_{\omega}$. Hence $-G'(0)$ belongs to $\rho(\Gamma)$. The proof is complete.
\end{proof}

Let us return to $R(-G'(0),\G)$ defined in \eqref{re} for a moment.   In this definition, the function $h$ is a certain spiral-like function with $h(0)=0$, as the Koenigs function of the semigroup $h$, and hence $h$ is certainly univalent on $\D$. Since the resolvent is always bounded, such $R(-G'(0),\G)$ is bounded on $A^p_\om$. Moreover, its compactness plays an essential role in characterizing the compactness of the general $R(\lambda,\G)$. Nevertheless, such an operator can be well-defined for any univalent function $h$ with $h(0)=0$. Therefore, it seems that it is even more natural to characterize bounded and compact $R_h$ on $A^p_\om$ with $\om\in\DD$ for the full range $0<p<\infty$, where $R_h$ is defined as
$$
R_h(f)(z)=\frac{1}{h(z)}\int_0^z f(\z)h'(\z)\,d\z,\quad f\in\H(\D),
$$
provided $h$ is univalent function on $\D$ with $h(0)=0$. 

Before presenting the result for the above question, we first recall that for $\alpha\geq1$ and $\omega\in\DD$, $\mathcal{C}^{\alpha}(\omega^\star)$ consists of all $f\in\H(\D)$ such that
$$
\|f\|^2_{\mathcal{C}^{\alpha}(\omega^\star)}=|g(0)|^2+
\sup_{I\in\T}\frac{\int_{S(I)}|f'(z)|^2\omega^\star(z)dA(z)}{(\omega(S(I)))^{\alpha}}<\infty;
$$
$\mathcal{C}_0^{\alpha}(\omega^\star)$ consists of all $f\in\H(\D)$ such that
$$
\lim_{|I|\rightarrow0}\frac{\int_{S(I)}|f'(z)|^2\omega^\star(z)dA(z)}{(\omega(S(I)))^{\alpha}}=0.
$$
If $\alpha=1$, then by \cite{PeRa2014} we have
\begin{equation}\label{relation}
  \BMOA\subset \mathcal{C}^{1}(\omega^\star)\subset 
  \B,\quad \VMOA\subset \mathcal{C}_0^{1}(\omega^\star)\subset 
  \B_0.
\end{equation}

Our conclusion can be stated as follows, which contains the results of Theorem \ref{theorem3}.
\begin{theorem}\label{thm}
Let $0< p<\infty$ and $\om\in\DD$. Suppose $h$ is a univalent function on $\D$ satisfying $h(0)=0$. Then $R_h$ is automatically bounded on $A^p_\om$. Moreover, the following statements are equivalent:
\begin{itemize}
    \item [(i)] $R_h$ is compact on $A^p_\om$;
   \item[(ii)] $\log\frac{h(z)}{z}\in\mathcal{C}^1_0(\omega^\star)$;
\item[(iii)] $\log\frac{h(z)}{z}\in\B_0$;
\item[(iv)] $\log\frac{h(z)}{z}\in\VMOA$.
\end{itemize}
\end{theorem}

\begin{proof}
  Since $h$ is univalent with $h(0)=0$ by the hypothesis,  the equivalence among (ii), (iii) and (iv) is obvious by the relation \eqref{relation} and a result in \cite{Pom}. Therefore, it suffices to verify that $R_h$ is always bounded and (i) and (ii) are equivalent. To this end, we are going to use a method applied in the proof of \cite[Theorem 6.1]{AG4}, which points out that $R_h$ can be decomposed as follows.
\begin{equation}\label{eq}
M_zP_h=R_hM_z, \quad Q_h=P_h+Q_hP_h,
\end{equation}
where
$$
M_zf(z)=zf(z),
$$
$$
P_hf(z)=\frac{1}{zh(z)}\int_0^zf(\zeta)\zeta h'(\zeta)d\zeta,
$$
and
$$
Q_hf(z)=\frac1z\int_0^zf(\zeta)\frac{\zeta h'(\zeta)}{h(\zeta)}d\zeta.
$$
Moreover, it is elementary to see that
\begin{equation}\label{eq1}
Q_hf(z)=J(f)(z)+L_hM_z(f)(z),
\end{equation}
where
$J(f)(z)=\frac{1}{z}\int_0^zf(\zeta)d\zeta$ and
\begin{equation}\label{eq2}
L_h(f)(z)=\frac{1}{z}\int_0^zf(\zeta)\left(\log\frac{h(\zeta)}{\zeta}\right)'d\zeta.
\end{equation}
Now, since $h$ is univalent with $h(0)=0$, it follows from \cite[Theorem 2]{Ba} that $\log\frac{h(z)}{z}\in\BMOA$ hence it certainly belongs to $\mathcal{C}^1(\omega^\star)$ by \eqref{relation}. Then applying the \cite[Theorem 6.8 (i)]{pel2016}, we immediately see that $L_h$ is bounded on $A^p_\om$. Moreover, trivially, $J$ is also bounded on $A^p_\om$. Thus, $Q_h$ is bounded on $A^p_\om$ by \eqref{eq1}, and so are $P_h$ and $R_h$ by the identities \eqref{eq}.

Similarly, to prove that (i) and (ii) are equivalent, we only need to notice that $L_h$ in \eqref{eq2} is compact if and only if $\log\frac{h(z)}{z}\in\mathcal{C}^1_0(\omega^\star)$. Then since $J$ is compact on $A^p_\om$, $L_h$ is compact if and only if $Q_h$ is compact by \eqref{eq1}. Likewise, by \eqref{eq}, $Q_h$ is compact if and only if $P_h$ is compact if and only if $R_h$ is compact. Combining those facts, we are in a position to finish the proof.
\end{proof}

As a corollary of the above theorem, we can get some characterizations for compact resolvent of composition semigroup action on $A^p_\om$.
\begin{corollary}\label{coroally}
Let $1\leq p<\infty$ and $\omega\in\widehat{\mathcal{D}}$. Suppose $\at$ is a semigroup of analytic self-maps of $\D$ with Denjoy-Wolff point $0$, infinitesimal generator $G$, and associated Koenigs function $h$. Denote by $\G$ the infinitesimal generator if the corresponding composition semigroup $\Ct$ on $A^p_{\omega}$ and denote by $R(\lambda,\G)$ the resolvent operator for $\lambda\in\rho(\G)$. Then the following statements are equivalent:
\begin{itemize}
\item[(i)] $R(\lambda,\G)$ is compact on $A^p_{\omega}$;
\item[(ii)] $\log\frac{h(z)}{z}\in\mathcal{C}^1_0(\omega^\star)$;
\item[(iii)] $\log\frac{h(z)}{z}\in\B_0$;
\item[(iv)] $\log\frac{h(z)}{z}\in\VMOA$;
\item[(v)] For any $\z\in\partial\D$, 
$$
\lim_{z\to\z}\left|\frac{G(z)}{z-\z}\right|=\infty.
$$
\end{itemize}
\end{corollary}

\begin{proof}
The well-known resolvent equation
$$
R(\lambda,\G)-R(\mu,\G)=(\mu-\lambda)R(\lambda,\G)R(\mu,\G) \quad \lambda,\mu\in\rho(\G)
$$
shows that $R(\lambda,\G)$ is compact on $A^p_{\omega}$ for all $\lambda\in\rho(\G)$ if and only if it is compact for a certain $\lambda_0\in\rho(\G)$. So, by Lemma \ref{lemma3}, to prove the compactness of $R(\lambda,\G)$, it suffices to identify the compactness of $R(-G'(0),\G)$, or equivalently, the compactness of $R_h$ defined as before. Therefore, a direct application of Theorem \ref{thm} gives the equivalence of (i), (ii), (iii) and (iv). In addition, the equivalence of (v) and (iii) follows from \cite[Theorem 6.1]{AG4}. The proof is complete.
\end{proof}

Another attempt for $R_h$ is to characterize its spectrum on $A^p_\om$ for any univalent function $h$ with $h(0)=0$. This question is generally challenging for those $h$, but it seems promising for some special $h$ due to the spectral mapping theorem for the resolvent if $h$ is a Koenigs function of some semigroups. 

First, recall that a univalent function $h:\D\rightarrow\Omega:=h(\D)$ is said to be $\mu$-spiral-like if there is a number $\mu\in\C$ with $\text{Re}\mu>0$ such that for each $\omega\in\Omega$ and $t\geq0$ the point $e^{-t\mu}\omega$ also belongs to $\Omega$. Then the class $\varphi_t(z):=h^{-1}(e^{-\mu t}h(z))$ becomes a one-parameter semigroup with the Koenigs function $h$.  Moreover, the Denjoy-Wolff point of such $\at$ is 0 and  infinitesimal generator $G=-\mu\frac{h(z)}{h'(z)}$ with $G'(0)=-\mu$. Therefore, by Lemma \ref{lemma3}, $\mu$ is in the resolvent set of $\G$, the generator of the corresponding $\Ct$, and the resolvent $R(\mu,\G)$ can be written as
\begin{equation}\label{1x}
 R(\mu,\G)=\frac{1}{\mu h(z)}\int_0^z f(\z)h'(\z)\,d\z=\frac{1}{\mu}R_h,\quad f\in A^p_\om.   
\end{equation}

Next, the following spectral mapping theorem for the resolvent (see \cite[p.161]{EN1}) asserts that 
\begin{equation}\label{2x}
 \sigma_{A^p_\om}(R(\mu,\G))\setminus\{0\}=(\mu-\sigma_{A^p_\om}(\G))^{-1}:=\left\{\f{1}{\mu-\lambda}:\lambda\in\sigma_{A^p_\om}(\G)\right\}   
\end{equation}
for each $\lambda\in\rho(\G)$. Combining those facts, we may get the following result:
\begin{proposition}\label{pro}
Let $1\leq p<\infty$, $\Re{\mu}>0$ and $\om\in\DD$. Suppose $h$ is a $\mu$-spiral-like function on $\D$ satisfying $h(0)=0$. Then the following statements hold:
\begin{itemize}
    \item [(i)] If $\log\frac{h(z)}{z}\in\B_0$, then 
    $$
    \sigma_{A^p_\om}(R_h)=\left\{\frac{1}{k+1}, ~h^k\in A^p_\om, ~k=0,1,2,\cdots\right\}\cup\{0\};
    $$
    \item [(ii)] For any $t>0$, $\sigma_{A^2_\om}(R_h)$ satisfies the following identity:
    $$
   \left\{e^{t\left(\mu-\frac{\mu}{\sigma_{A^2_\om}(R_h)}\right)}: \lambda\in\sigma(\G)\right\}\cup\{0\}=\{\lambda\in\C:|\lambda|\leq r_{e, A^2_\om}(C_{\varphi_{t}})\}\cup\{e^{-k\mu t}:k=0,1,2,3\ldots\},
    $$
\end{itemize}
where $\varphi_t(z):=h^{-1}(e^{-\mu t}h(z))$ for all $z\in\D$ and $\G$ is the generator of  the corresponding composition semigroup.
\end{proposition}

\begin{proof}
    (i) By the above 
explanation, we see that $\varphi_t(z)=h^{-1}(e^{-\mu t}h(z))$ is a one-parameter semigroup of the self-maps with the Denjoy-Wolff point 0. Let $\G$ be the generator of the corresponding composition semigroup on $A^p_\om$. 
    Since $\log\frac{h(z)}{z}\in\B_0$, by Corollary \ref{coroally}, we see that $R(\mu,\G)$ is compact on $A^p_\om$ and hence its spectrum only consists of point spectrum. Therefore, these facts together with Theorem \ref{theorem1} and \eqref{2x} yield
    \begin{equation}\label{xf}
    \sigma_{A^p_\om}(R(\mu,\Gamma))\setminus\{0\}=\left\{\frac{1}{\mu-kG'(0)}:~h^k\in A^p_\om\right\}=\left\{\frac{1}{(k+1)\mu}:~h^k\in A^p_\om\right\}. 
   \end{equation}
   In addition, \eqref{1x} yields $R_h=\mu R(\mu,\G)$ and hence $\sigma_{A^p_\om(R_h)}=\mu \sigma_{A^p_\om}(R(\mu,\Gamma))$ and we have $0\in\sigma_{A^p_\om(R_h)}$ since $R_h(f)(0)=0$ for any $f\in\H(\D)$, which combined with \eqref{xf} concludes the result.

   (ii) Since we focus on the Hilbert space $A^2_\om$, the spectral mapping theorem holds automatically. Then the combined application of Theorem \ref{theorem1}, \eqref{1x} and \eqref{2x} completes the proof. 
\end{proof}

\vskip4mm
We finish the section and the paper by presenting several examples, while other constructions certainly can be provided. 

\begin{example}
    Let $h(z)=z$. Then $h$ is a $c$-spiral-like function for any $c$ with $\Re c>0$. By Proposition \ref{pro}, we get
    $$
\sigma_{A^p_\om}(R_z)=\left\{\frac{1}{k+1}, k=0,1,2,\cdots\right\}\cup\{0\}.
    $$
\end{example}

\begin{example}
    Let $h(z)=\log\frac{1}{1-z}$. Then $h$ is a $1$-spiral-like function. It induces the one-parameter semigroup $\at$ with $\varphi_t(z)=1-(1-z)e^{-t}$. By Proposition \ref{pro}, we get
    $$
\sigma_{A^p_\om}(R_z)=\left\{\frac{1}{k+1}, \left(\log\frac{1}{1-z}\right)^k\in A^p_\om, k=0,1,2,\cdots\right\}\cup\{0\}.
    $$
\end{example}

\begin{example}
    Let $h(z)=\log\frac{1+z}{1-z}$. Then $h$ is a $1$-spiral-like function. It induces the one-parameter semigroup $\at$ with $\varphi_t(z)=\frac{(1+z)^{e^{-t}}-(1-z)^{e^{-t}}}{(1+z)^{e^{-t}}+(1-z)^{e^{-t}}}$. We have
    $$
\sigma_{A^p_\om}(R_z)=\left\{\frac{1}{k+1}, \left(\log\frac{1+z}{1-z}\right)^k\in A^p_\om, k=0,1,2,\cdots\right\}\cup\{0\}.
    $$
\end{example}

\vskip1cm

{\bf{Acknowledgements}}~~~The authors thank the anonymous referee for careful reading of the manuscript and for his/her insightful comments and constructive suggestions which improved the exposition of the paper and certainly helped to make it more readable. The authors would like to express their gratitude to Professor A. G. Siskakis for his valuable suggestion on Theorem 1, which inspires us to present Proposition \ref{proo} and Corollary \ref{theorem+}.

\end{document}